\documentclass[11pt]{amsart}
\usepackage{amsmath,amssymb,amsbsy,amsfonts,amsthm,latexsym,
                       amsopn,amstext,amsxtra,euscript,amscd}

\usepackage{array}
\usepackage{ifthen}
\usepackage{url}

\newtheorem{thm}{Theorem}

\newtheorem{lemma}{Lemma}

\newtheorem*{conjecture}{Conjecture}

\theoremstyle{definition}

\newcommand{\klammern}[4][]%
{\ifthenelse{\equal{#1}{}}{\left#2}{\csname#1\endcsname#2}%
#4\ifthenelse{\equal{#1}{}}{\right#3}{\csname#1\endcsname#3}}

\begin{document}

\title{On Diophantine quintuple Conjecture}
\author{{Wenquan Wu and Bo He}}


\thanks{Supported by National Natural Science Foundation of China (Grant No. 11301363)}

\maketitle
\begin{abstract}
In this note, we prove that if $\{a,b,c,d,e\}$ with $a<b<c<d<e$ is a Diophantine quintuple, then $d<10^{76}$.
\end{abstract}

A set of $m$ distinct positive integers $\{a_1, \dots,  a_m\}$ is
called a Diophantine m-tuple  if $a_i a_j +1$ is a perfect square. Diophantus studied sets of positive rational numbers with the same property, particularly he found the set of four positive rational numbers $\left\{\frac{1}{16}, \frac{33}{16}, \frac{17}{4}, \frac{105}{16}\right\}$. But the first Diophantine quadruple was found by Fermat. In fact, Fermat proved that the set $\{1, 3, 8, 120\}$ is a Diophantine quadruple called {\it Fermat set}. Moreover, Baker and Davenport \cite{Baker-Davenport:1969} proved that the set $\{1, 3, 8, 120\}$ cannot be extended to a Diophantine quintuple.

 Several results of the generalization of the result of Baker and Davenport are obtained. In 1997, Dujella \cite{Dujella:1997-1} proved that the Diophantine triples of the form $\{k-1, k+1, 4k\}$, for $k \geq 2$, cannot be extended to a Diophantine quintuple.  The Baker-Davenport's result corresponds to $k=2$. In 1998, Dujella and Peth\"{o} \cite{Dujella-Pethoe:1998} proved that the Diophantine pair $\{1, 3\}$ cannot be extended to a Diophantine quintuple. In 2008, Fujita \cite{Fujita} obtained a more general result by proving that the Diophantine pairs $\{k - 1, k + 1\}$, for $k \geq 2$, cannot be extended to a Diophantine quintuple. A folklore conjecture is
\begin{conjecture}
There does not exist a Diophantine quintuple.
\end{conjecture}
 In 2004, Dujella \cite{Dujella:2004} proved that there are
only finitely many Diophantine quintuples. Assuming that
$\{a, b, c, d, e\}$ is a Diophantine quintuple with $a<b<c<d<e$, authors got the upper bound of element $d$:

i) $d<10^{2171}$, Dujella \cite{Dujella:2004} ;

ii) $d<10^{830}$, Fujita \cite{Fujita-number};

iii)  $d<10^{100}$, Filipin and Fujita \cite{Filipin-Fujita}.

iv) $d<3.5\cdot 10^{94}$, Elsholtz, Filipin and Fujita \cite{Elsholtz}.

Moreover, by using upper bound of $d$, corresponding upper bound of
number of Diophantine quintuples are obtained, $10^{1930}$,
$10^{276}$, $10^{96}$ and $6.8\cdot 10^{32}$respectively.

In this paper, we prove the following results.
\begin{thm}\label{thm:1}
If $\{a,b,c,d,e\}$ is a Diophantine quintuple with $a<b<c<d<e$, then $d<10^{76}$.
\end{thm}

From now on, we will assume that $\{a,b,c,d,e\}$ is a Diophantine
quintuple with $a<b<c<d<e$.  Let us consider a Diophantine triple
$\{A, B, C\}$. We define the positive integers $R,S,T$ by
\begin{equation}\label{eq:def-rst}
\notag AB+1=R^2,\quad AC+1=S^2,\quad BC+1=T^2.
\end{equation}
 In order to extend the Diophantine triple $\{A,B,C\}$ to a Diophantine quadruple $\{A,B,C,D\}$, we have to solve the system
\begin{equation}\label{eq:system-d}
\notag    AD+1=x^2,\quad BD+1=y^2,\quad CD+1=z^2,
\end{equation}
in integers $x, y, z$. Eliminating $D$, we obtain the following
system of Pellian equations.
\begin{eqnarray}
\label{eq:system-zx}     Az^2-Cx^2 &=& A-C,\\
\label{eq:system-zy} Bz^2-Cy^2 &=& B-C.
\end{eqnarray}
All solutions of \eqref{eq:system-zx} and \eqref{eq:system-zy} are
respectively given by $z=v_m$ and $z=w_n$ for some integer $m,n \ge
0$, where
$$
v_0=z_0,\quad v_1=Sz_0+Cx_0,\quad v_{m+2}=2Sv_{m+1}-v_m,
$$
$$
w_0=z_1,\quad w_1=Tz_1+Cy_1,\quad w_{n+2}=2Tw_{n+1}-w_n,
$$
with some integers $z_0,z_1,x_0,y_1$.

\vspace{5mm}

 By Lemma~3 of \cite{Dujella:2004}, we
have the following relations between $m$ and $n$.
\begin{lemma}\label{lem:2}
If $v_{2m} = w_{2n}$, then $n\le m \le 2n$.
\end{lemma}

We will give a new lower bound of $m$.
this paper.
\begin{lemma}\label{lem:3}
If $B\ge 8$ and $v_{2m} = w_{2n}$ has solutions $m\ge3, n\ge 2$, then
$m>0.48B^{-1/2}C^{1/2}$.
\end{lemma}
\begin{proof}
By Lemma 4 in \cite{Dujella:2001} and $z_0=z_1=\lambda\in\{1,-1\}$,
we have
$$
Am^2 + \lambda Sm \equiv Bn^2 + \lambda Tn \pmod{4C}.
$$
Suppose that $m\le 0.48 B^{-1/2}C^{1/2}$. From the relation $n\le m$, we get
$$
\max\{Am^2,Bn^2\}\le Bm^2  \le 0.25 B\cdot B^{-1}C< 0.25C
$$
and
$$
\max\{Sm, Tn\} \le Tm <0.48
(BC+1)^{1/2}B^{-1/2}C^{1/2}<0.5(BC)^{1/2}B^{- 1/2}C^{1/2}=0.5C.
$$
We obtain that
$$
Am^2-Bn^2 = \lambda (Tn-Sm).
$$
This implies
$$
\lambda (Tn+Sm)(Am^2-Bn^2) = T^2n^2 - S^2m^2
$$
$$
= (BC+1)n^2 - (AC+1)m^2
= C(Bn^2 - Am^2) + n^2 -m^2.
$$
It follows that
$$
m^2 - n^2 = (C+\lambda(Tn+Sm))(Bn^2 - Am^2).
$$
If $Bn^2 - Am^2 =0$, then $m=n$, it is impossible. Hence,
$$
m^2 - n^2 = |m^2-n^2| \ge |C+\lambda(Tn+Sm)|.
$$
The case $\lambda=1$ provides $m^2 > C$, it is a contradiction to
$m<0.48B^{-1/2}C^{1/2}$. From $Tn+Sm<2Tn<C$, we need to consider
$$
m^2 - n^2 = |m^2-n^2| \ge |C-(Tn+Sm)|=C-(Tn+Sm).
$$
Therefore, we get the inequality
$$C\le Tn+Sm+m^2-n^2\le2Tm + 0.75m^2$$
$$
<0.96(BC+1)^{1/2}B^{-1/2}C^{1/2} + 0.173 B^{-1}C<C
$$
when $B\ge 8$. We have a contradiction. This completes the proof.
\end{proof}

\vspace{2mm}

\begin{proof}[\textbf{Proof of Theorem~\ref{thm:1}}]
Assume that $\{a,b,c,d,e\}$ is a Diophantine quintuple with $a<b<c<d<e$.
In \cite{Dujella-Pethoe:1998}, Dujella and Peth\"{o} have shown that the pair $\{1,3\}$ cannot extend to a Diophantine quintuple. This help us to assume that $b\ge 8$.

We choose that
$$
A=a,\,\, B=b,\,\, C=d,\,\, D=e
$$
in the Diophantine
quintuple $\{a,b,c,d,e\}$.
This implies the system of Pellian equations
\eqref{eq:system-zx} and \eqref{eq:system-zy} has positive integer solution $(x,y,z)$ with $|z|>1$. Equivalently, there are positive integers $j$ and $k$ satisfying $v_j=w_k$. By Lemma~5 and Lemma~6 of \cite{Fujita-number}, we have $j\equiv k\equiv 0 \pmod 2$, $k\ge4$ $z_0=z_1=\pm 1$. We set $j=2m$ and $k=2n$. Using Lemma~\ref{lem:3}, we have $m\ge 0.48B^{-1/2}C^{1/2}$.

It is known that $d\ge d^{+}> 4abc>4b^2$, where
$ d^{+} = a + b + c + 2abc + 2rst$. It results $B=b<d^{1/2}/2=C^{1/2}/2$. Hence, we have
\begin{equation}\label{eq:lowerbd-m}
  m\ge 0.678C^{1/4}.
\end{equation}

On the other hand, by used Theorem 2.1 in \cite{Matveev:2000} of Matveev ,
we have the relative upper bound (cf. Proposition 14 of
\cite{Fujita-number})
\begin{equation}\label{eq:ubd-m}
\frac{m}{\log(351m)}<2.786\cdot 10^{12}\cdot \log^2 C.
\end{equation}
Combining \eqref{eq:lowerbd-m} and \eqref{eq:ubd-m}, we obtain
$$
C^{1/4}<4.11\cdot 10^{12} \cdot \log^2 C \cdot \log(238C^{1/4}).
$$
Therefore, we have $d=C<10^{76}$. This complete Theorem~\ref{thm:1}.
\end{proof}

$\;$

Wenquan Wu

Department of Mathematics

Aba Teacher's College

Wenchuan, Sichuan 623000

P. R. China

wwq681118@163.com

$\;$

 Bo He

Department of Mathematics

Aba Teacher's College

Wenchuan, Sichuan 623000

P. R. China

bhe@live.cn


\begin{thebibliography}{1}


\bibitem{Baker-Davenport:1969} A.~Baker and H.~Davenport, \emph{The equations {$3x^2-2=y^2$} and {$8x^2-7=z^2$}}, Quart. J.~Math. Oxford \textbf{20} (1969), 129--137.



\bibitem{Dujella:1997-1} A. Dujella, \emph{The problem of the extension of a parametric family of Diophantine triples}, Publ. Math. Debrecen \textbf{51} (1997), 311--322.



\bibitem{Dujella:2001} A. Dujella, \emph{An absolute bound for the size of Diophantine $m$-tuples}, J. Number Theory, {\bf 89} (2001) 126--150.


\bibitem{Dujella-Pethoe:1998} A. Dujella and A. Peth\"{o}, \emph{A generalization of a theorem of Baker and Davenport}, Quart. J. Math. Oxford Ser. (2) \textbf{49} (1998), 291--306.

\bibitem{Dujella:2004} A. Dujella, \emph{There are only finitely many Diophantine quintuples}, J. Reine Angew. Math. \textbf{566} (2004), 183--214.

\bibitem{Dujella:2008} A. Dujella, \emph{On the number of Diophantine m-tuples}, Ramanujan J. \textbf{15} (2008),37--46.










\bibitem{Elsholtz} C. Elsholtz, A. Filipin and Y. Fujita \emph{On Diophantine quintuples and $D(-1)$-quadruples}, Monatsh. Math., to appear.

\bibitem{Fujita} Y. Fujita, \emph{The extensibility of Diophantine pairs $\{k-1, k+1\}$}, J. Number Theory, \textbf{128} (2008), 322--353.



\bibitem{Fujita-number} Y. Fujita, \emph{The number of Diophantine quintuples}, Glas. Mat. Ser. III \textbf{45} (2010), 15--29.

\bibitem{Filipin-Fujita} A. Filipin and Y. Fujita, \emph{The number of Diophantine quintuples II}, Publ. Math. Debrecen,
\textbf{82} (2013), 293--308.











\bibitem{Matveev:2000} E. M.~Matveev, \emph{An explicit lower bound for a homogeneous rational linear
form in logarithms of algebraic numbers II}, Izv.~Math \textbf{64}
(2000), 1217--1269.




















\end{thebibliography}
\end{document}